\documentclass[leqno,12pt]{amsart}
\usepackage{amssymb,amsthm,amsmath,txfonts,a4wide}
\usepackage{xcolor}
\usepackage{enumerate}
\usepackage{hyperref}
\usepackage{tikz}
\usetikzlibrary{shapes,arrows}
\usepackage{times}
\usepackage{verbatim}

\hypersetup{
colorlinks,
citecolor=green,
filecolor=blue,
linkcolor=red,
urlcolor=blue
}

\newtheorem{defn}{Definition}[section]
\newtheorem{lemma}[defn]{Lemma}

\newtheorem{theorem}[defn]{Theorem}

\newtheorem{proposition}[defn]{Proposition}
\newtheorem{condition}[defn]{Condition}

\numberwithin{equation}{section}

\begin{document}
	
\title{Morawetz estimates and stabilization for damped Klein-Gordon equation with small data}
\author{Yan Cui}
\thanks{Yan Cui is supported by Science and Technology Projects in Guangzhou (No. 2023A04J1335), by the Open Project Program of Shanghai Key Laboratory for Contemporary Applied Mathematics (No. SKLCAM202403002)}
\address{Department of Mathematics, Jinan University, Guangzhou, P. R. China.}
\email{cuiy32@jnu.edu.cn}

\author{Bo Xia}
\address{School of Mathematical Sciences,USTC, Hefei, P. R. China.}
\email{xiabomath@ustc.edu.cn,xaboustc@hotmail.com}
\thanks{Bo Xia was supported by NSFC No. 12171446.}

\begin{abstract}
	In the present paper, we show that the global solution to (partially) damped Klein-Gordon equation on the three dimensional Euclidean space with small data decays exponentially. The key ingredients in the proof are: Morawetz-type estimates for solutions with small data and Ruiz's unique continuation principle for wave equations. 
\end{abstract}
\maketitle


\section{Introduction and main results}
	In this article, we will consider the stabilization problem for damped Klein-Gordon equation
	\begin{equation}\label{eq:kg:int}
		\left\{
			\begin{split}
				\partial^2_{tt}u(t,x)-\Delta u(t,x)&+\alpha(x)\partial_tu(t,x)+u(t,x)=u^3(t,x)\\
				\big(u(0,x),\partial_tu(0,x)\big)&=\big(u_0(x),u_1(x)\big)
			\end{split}
		\right.\ \ \ \ \ \  (t,x)\in\mathbb{R}\times\mathbb{R}^3
	\end{equation}
	in the energy space with small initial data. Throughout the whole paper, our assumption on the damping term $\alpha(x)$ is
	\begin{condition}\label{condition:1}
		There exist positive constants $\Lambda_1,\Lambda_0$ and $R$ such that the nonnegative function $\alpha\in L^\infty(\mathbb{R}^3)$ satisfies
		\begin{equation*}
			\Lambda_0\leq \alpha(x)\leq \Lambda_1,\ \forall x\in\mathbb{R}^3\backslash B_R
		\end{equation*}
		where $B_R$ is the ball centered at the origin with radius $R$ in $\mathbb{R}^3$.
	\end{condition}
	
	The stationary solution of \eqref{eq:kg:int} obeys
	\begin{equation}\label{eq:kg:sta}
		-\Delta\varphi +\varphi =\varphi^3\ \ \mathrm{on}\ \mathbb{R}^3.
	\end{equation}
	This equation admits a unique positive solution denoted by $Q$ (see \cite{Coffman1972}). We now introduce the variational characterization of $Q$. For any $\varphi\in H^1(\mathbb{R}^3)$, denote
	\begin{equation*}
		J[\varphi]:=\int_{\mathbb{R}^3}\left[\frac{|\nabla\varphi|^2+\varphi^2}{2}-\frac{\varphi^4}{4}\right]dx
	\end{equation*}
	and
	\begin{equation*}
		K[\varphi]:=\int_{\mathbb{R}^3}\left( |\nabla\varphi|^2+\varphi^2 -\varphi^4\right)dx.
	\end{equation*} 
	Then the infimum 
	\begin{equation*}
		h_0:=\inf\left\{J[\varphi]:\varphi\in H^1\backslash\{0\}, K[\varphi]=0 \right\}
	\end{equation*} 
	is strictly positive and it is attained by $Q$ (see \cite[Section 2.2]{nakanishischlag2011}).
	
	For a solution $u$ of \eqref{eq:kg:int}, we denote $\vec{u}=(u,\partial_tu)$ and define the energy to be
	\begin{equation*}
		E[\vec{u}(t)]:=\int_{\mathbb{R}^3}\left[\frac{|\nabla_{t,x}u(t,x)|^2+u(t,x)^2}{2}-\frac{u(t,x)^4}{4}\right]dx
	\end{equation*} 
	and the linear energy to be 
	\begin{equation*}
			E_L[\vec{u}(t)]:=\int_{\mathbb{R}^3}\frac{|\nabla_{t,x}u(t,x)|^2+u(t,x)^2}{2}dx.
	\end{equation*}
	
	Our stabilization result begins with solving \eqref{eq:kg:int}.
	\begin{proposition}\label{thm:1}
		Under Condition \ref{condition:1}, the following two sets
		\begin{equation*}
			\mathcal{PS}^+:=\left\{(u_0,u_1)\in H^1\times L^2: E[(u_0,u_1)]<h_0,K[u]\geq 0 \right\}
		\end{equation*}
		and
		\begin{equation*}
			\mathcal{PS}^+:=\left\{(u_0,u_1)\in H^1\times L^2: E[(u_0,u_1)]<h_0,K[u]< 0 \right\}
		\end{equation*}
		are invariant under the flow of \eqref{eq:kg:int} as long as the flow is defined. What's more, one has the following dichotomy of dynamics
		\begin{itemize}
			\item the solution starting from $\mathcal{PS}^+$ exists for all time;
			\item the solution starting from $\mathcal{PS}^-$ blows up in finite time.
		\end{itemize}
	\end{proposition}
	This result can be proved by adapting Payne-Sattinger argument to \eqref{eq:kg:int}. One can refer to \cite{nakanishischlag2011} for a standard introduction about this argument in dealing with Klein-Gordon equation over the whole Euclidean space. See \cite{cui2024} for another adaption of this argument to a system of damped Klein-Gordon equations. 
	
	Let $u$ be a global solution of \eqref{eq:kg:int} with data in $\mathcal{PS}^+$. For each time $t$, we write
	\begin{equation*}
		E[\vec{u}(t)]=\frac{1}{4}K[u(t)]+\frac{1}{2}E_L[\vec{u}(t)] +\frac{1}{4}\int_{\mathbb{R}^3}|\partial_{t}u|^2dx
	\end{equation*}
	which implies immediately
	\begin{equation}\label{eq:E:El}
		E_L[\vec{u}(t)]\leq 2E[\vec{u}(t)].
	\end{equation}
	What's more, using the multiplier method, we have the identity for each time $T>0$
	\begin{equation}\label{eq:id}
		E[\vec{u}(T)]=E[\vec{u}(0)]-A[\vec{u},0,T]
	\end{equation}
	where the energy decrement $A[\vec{u},0,T]$ is given as
	\begin{equation}
		A[\vec{u},0,T]:=\int_0^T\int_{\mathbb{R}^3}\alpha(x)|\partial_tu|^2dx.
	\end{equation}
	In particular, this identity implies that the energy of $u$ is non-increasing in time. 
	In the present paper, we show in a further step that the energy of this long-time solution $u$ with small data not only decays, but also it decays exponentially.
	  	
	\begin{theorem}\label{thm:main} 
		Under Condition \ref{condition:1}, there exists a small number $\epsilon>0$ such that any global solution, starting from $\mathcal{PS}^+$ with initial data satisfying $\|(u_0,u_1)\|_{H^1\times L^2}<\epsilon$, decays exponentially.
	\end{theorem}
	
	{Although this result is stated in \cite{ain11}, their proof does not work in the present case: the global solution in $\mathcal{PS}^+$ with small data. Since we did not find any proof in some other place neither, we present the proof explicitly here, which turns out not be an easy task. The main difficulty arises from the focusing nature of \eqref{eq:kg:int}. This seems to be even harder in the case of large data, for which one can find in \cite{cui2024} a concrete treatment, including the derivation of Morawetz-type inequality and the demonstration of observation inequality for a system of damped Klein-Gordon equations.}
	
	Theorem \ref{thm:main} follows by applying the time-translation invariance of \eqref{eq:kg:int} to the following observation inequality (see \cite[Section 6]{cui2024} for the treatment of s system of damped Klein-Gordon equations).
	\begin{theorem}\label{thm:main:0}Assume that Condition \ref{condition:1} holds so that constants $\Lambda_0,\Lambda_1$ and $R$ are given.  
		{Then for each sufficiently small number $\epsilon>0$}, there exists a time $T=T(\Lambda_1,\Lambda_0,R,\epsilon)$ and a constant $C=C(\Lambda_1,\Lambda_0,R,\epsilon)$ such that any global solution $u$ of \eqref{eq:kg:int} with data $\vec{u}(0)=(u_0,u_1)\in\mathcal{PS}^+$ satisfies the estimate
		\begin{equation*}
			E[\vec{u}(T)]\leq CA[\vec{u};0,T]
		\end{equation*}
		provided that
		\begin{equation*}
			E_L[\vec{u}(0)]<\epsilon.
		\end{equation*}
	\end{theorem}
	
	The key ingredients in proving the observation inequality are Ruiz's unique continuation principle (see \cite{ruiz92}) and the following weak version of observation inequality (see \cite[Section 6]{cui2024} for a concrete treatment for a system of damped Klein-Gordon equations).
	
	\begin{proposition}\label{prop:obs:weak}
		Under the same assumption in Theorem \ref{thm:main:0}, there exists a positive number $\epsilon_0$ satisfying the following property. For each $\epsilon\in(0,\epsilon_0)$, there exist a time $T_0=T_0(\Lambda_1,\Lambda_0,R,\epsilon_0)$ and a constant $C_0=C_0(\Lambda_1,\Lambda_0,R,\epsilon_0)$ such that any global solution $u$ of \eqref{eq:kg:int} with $\vec{u}(0)=(u_0,u_1)\in\mathcal{PS}^+$ satisfies {for each time $T\geq T_0$}
		\begin{equation}\label{eq:obs:week}
			E[\vec{u}(T)]\leq C_0\left[A[\vec{u};0,T]+\int_0^{T}\left\|u\right\|^2_{L^2(|\cdot|\leq 4R)}dt\right]
		\end{equation}
		provided that 
		\begin{equation*}
			E_L[\vec{u}(0)]<\epsilon.
		\end{equation*}
	\end{proposition}
	
	The proof of this weak version of observation inequality is a consequence of variants of Morawetz estimate for \eqref{eq:kg:int} with small data.
	\begin{lemma}\label{lem:morawetz}
		Under the same condition with Proposition \ref{prop:obs:weak}, there exists a positive number $\epsilon_2$ such that for each $\epsilon\in(0,\epsilon_2)$, the following property holds. One can finds a constant $C_2=C_2(\Lambda_1,\Lambda_0,R,\epsilon_2)$ so that any global solution $u$ with data $\vec{u}(0)=(u_0,u_1)\in\mathcal{PS}^+$ satisfies for each time $T>0$
		\begin{equation}\label{eq:u4:1}
			\int_0^T\int_{|\cdot|\leq 2R}u^4dxdt\leq C_2\epsilon \left[\int_0^T\int_{|\cdot|\leq 4R} u^2dxdt + \int_0^T\int_{2R\leq |\cdot|\leq 4R}|\nabla u|^2dxdt \right]
		\end{equation}
		
		\begin{equation}\label{eq:u:nabla:1}
			\int_0^T\int_{|\cdot|\leq 2R}|\nabla u|^2dxdt\leq C_2\left[A[\vec{u};0,T]+E[\vec{u}(T)] + \int_0^T\int_{|\cdot|\leq 4R} u^2dxdt + \int_0^T\int_{2R\leq |\cdot|\leq 4R}|\nabla u|^2dxdt \right]
		\end{equation}
		
		\begin{equation}\label{eq:u4:2}
			\int_0^T\int_{|\cdot|\geq 2R}u^4dxdt\leq C_2\epsilon \left[\int_0^T\int_{|\cdot|\leq 2R} u^2dxdt + \int_0^T\int_{2R\leq |\cdot|}|\nabla u|^2dxdt \right]
		\end{equation}
		
		\begin{equation}\label{eq:u:nabla:2}
			\begin{split}
				\int_0^T\int_{|\cdot|\geq 2R}\left(|\nabla u|^2+u^2\right)dxdt &\leq C_2\left[A[\vec{u};0,T]+E[\vec{u}(T)] + \int_0^T\int_{|\cdot|\leq 2R} u^2dxdt \right]\\
				&\ \ \ \ \ \ \ \ \ \ \ \ \ \ + C_2\epsilon \int_0^T\int_{R\leq |\cdot|\leq 2R}|\nabla u|^2dxdt 
			\end{split}
		\end{equation}
		provided that
		\begin{equation*}
			E_L[\vec{u}(0)]<\epsilon.
		\end{equation*}
	\end{lemma}
	
	With this result at hand, we now give
\begin{proof}[Proof of Proposition \ref{prop:obs:weak}]
	Let $\epsilon_2$ and $C_2$ be as in Lemma \ref{lem:morawetz}. Take a small number $\epsilon_0\in(0,\epsilon_2)$ that is to be fixed later on. Fix $\epsilon\in(0,\epsilon_0)$. Then any global solution $u$ of \eqref{eq:kg:int} with $\vec{u}(0)=(u_0,u_1)\in\mathcal{PS}^+$ satisfies estimates \eqref{eq:u:nabla:1} and \eqref{eq:u:nabla:2} with constant being $C_2$ for each $T>0$ provided that
	\begin{equation}\label{eq:data:assumption:1}
		E_L[\vec{u}(0)]<\epsilon<\epsilon_0.
	\end{equation}

	We add the inequality \eqref{eq:u:nabla:1} to $(C_2+1)$-multiple of \eqref{eq:u:nabla:2} and add $\int_0^T\int_{|\cdot|\leq 2R}u^2dxdt$ onto both sides of the resulted inequality, obtaining
	\begin{equation}\label{eq:energy:0}
		\begin{split}
			\int_0^T\int_{\mathbb{R}^3}(|\nabla u|^2+u^2)dxdt&\leq C_3\left[A[\vec{u};0,T]+E[\vec{u}(T)]\right]+  C_4\int_0^T\int_{|\cdot|\leq 4R} u^2dxdt\\
			&\ \ \ \ \ \ \ \ \ \ \ \ \ \ + C_2(C_2+1)\epsilon \int_0^T\int_{R\leq |\cdot|\leq 2R}|\nabla u|^2dxdt 
		\end{split}
	\end{equation}
	with $C_3:=C_2(C_2+2)$ and $C_4:=[C_2(C_2+2)+1]$. Taking $\epsilon_0$ to be small so that 
	\begin{equation*}
		C_2(C_2+1)\epsilon_0\leq \frac{1}{2},\ \mathrm{and\ hence}\ C_2(C_2+1)\epsilon\leq \frac{1}{2}
	\end{equation*}
	we can absorb the last term on the right hand side of \eqref{eq:energy:0} into its left hand side, obtaining
	\begin{equation*}
		\int_0^T\int_{\mathbb{R}^3}E_L[\vec{u}]dxdt\leq 2C_3\left[A[\vec{u};0,T]+E[\vec{u}(T)]\right]+  2C_4\int_0^T\int_{|\cdot|\leq 4R} u^2dxdt.
	\end{equation*}
	Since $E_L[\vec{u}(t)]\geq E[\vec{u}(t)]$ for each $t$ and the functional $E[\vec{u}(t)]$ is non-increasing in time, this last estimate implies				 
	\begin{equation*}
		TE[\vec{u}(T)]\leq 2C_3\left[A[\vec{u};0,T]+E[\vec{u}(T)]\right]+  2C_4\int_0^T\int_{|\cdot|\leq 4R} u^2dxdt.
	\end{equation*}
	Taking $T_0:=2C_3+1$ and kicking the term $2C_3E[\vec{u}(T)]$ back to the left hand side, we obtain the asserted estimate with $C_0:=2C_4$. This completes the proof of Proposition \ref{prop:obs:weak}.
\end{proof}

	For reader's convenience, we quote here Ruiz's unique continuation principle (u.c.p., for abbreviation) in \cite{ruiz92} with the underlying domain being balls.
	
\begin{theorem}[Ruiz' u.c.p.]\label{thm:ruiz}
	Fix a positive number $r>0$ and let $S\subset B_r$ be a non-empty open neighborhood of the boundary $\partial B_r$. For each time $T>\mathrm{diam}(B_r)=2r$, let $u$ be a week $L^2((0,T)\times B_r)$ solution of 
	\begin{equation*}
		(\partial^2_{tt}-\Delta )w+V(t,x)w=0\ \ \mathrm{in}\ (0,T)\times B_r
	\end{equation*}
	with $V(t,x)\in L^\infty((0,T),L^3(B_r))$. If in addition $w\equiv0$ on $(0,T)\times S$, then $w\equiv0$ on $(0,T)\times B_r$.
\end{theorem}

	We end this introduction part by describing the organization of this paper: in Section \ref{sec:morawetz}, we will prove Lemma \ref{lem:morawetz}, and in Section \ref{sec:main}, we will use Ruiz's u.c.p. and the weak version of observation inequality to prove the observation inequality.

\section{Proof of Lemma \ref{lem:morawetz}}\label{sec:morawetz}
	In this section, we give the proof of Lemma \ref{lem:morawetz}. In the defocusing case, that is, both the kinetic energy and the potential energy are of positive signs, the proof turns out to be simple, see \cite{ain11}. However, since our equation \eqref{eq:kg:int} is focusing, we can not apply the argument in \cite{ain11} to achieve these estimates. Nevertheless, we can handle this difficulty by exploring the smallness of linear energy.
	
	\begin{proof}[Proof of Lemma \ref{lem:morawetz}]
		We now show the first two inequalities. Let $r_1:=3R/2$ and $r_2:=5R/2$. Denote $S_{T,r}:=(0,T)\times B_r$ for each $r>0$ and each time $T>0$. Take a non-negative test function $\varphi\in C^2_c(B_{r_2},[0,1])$ that satisfies $\varphi\big|_{B_{r_1}}\equiv 1$ and $\sup_{B_{r_2}}|\nabla\varphi|\leq \gamma \varphi^{\frac{1}{2}} $ for some constant $\gamma=\gamma(R)$. Put $\phi=\varphi^4$. Then we have
		\begin{equation}\label{eq:constant}
			\max\big(\sup\left|x\phi(x)\right|,\sup(|x|\cdot|\nabla\phi|,|\nabla \varphi|,|\nabla \varphi^{\frac{7}{8}}|)\big)\leq \tilde{C}_1:=\tilde{C}_1(R),\ \sup_{B_{r_2}}|\nabla  \phi|\leq 4\gamma \phi^{\frac{7}{8}}.
		\end{equation} 
		For convenience, we here take $\tilde{C}_1$ to be larger so that it is bigger than $1$.
		
		Take $\epsilon_2<\epsilon_3$ to be two small positive numbers that are to be specified later on, and fix a number $\epsilon\in(0,\epsilon_2)$. Let $u(t)$ be a global solution of \eqref{eq:kg:int} with initial data $\vec{u}(0)=(u_0,u_1)\in\mathcal{PS}^+$ satisfying
		\begin{equation}\label{eq:data:assumption}
			E_L[\vec{u}(0)]<\epsilon<\epsilon_2<\epsilon_3.
		\end{equation}
		Integrating the product of the equation satisfied of $u$ and $\left(\phi x\cdot\nabla u+\phi u\right)$ over $S_{T,r_2}$ and using integration by parts, we obtain the identity 
		\begin{equation}\label{eq:id:1}
			\begin{split}
				0
				=&\underbrace{\iint_{S_{T,r_2}}\partial_{t}[\partial_{t}u(\phi x\cdot \nabla u+\phi u)]}_{=:\mathrm{I}} +\underbrace{ \iint_{S_{T,r_2}}\alpha \partial_t u (\phi x\cdot \nabla u+\phi u)}_{=:\mathrm{II}} \\
				&\underbrace{-\iint_{S_{T,r_2}}\nabla\cdot\left(\nabla u (\phi x\cdot \nabla u+\phi u)-\frac{\phi x}{2}\left(|\nabla u|^2+u^2-|\partial_tu|^2\right)\right) -\iint_{S_{T,r_2}}\nabla\cdot \frac{\phi xu^4}{4}}_{=:\mathrm{III}}  \\
				& +\underbrace{\iint_{S_{T,r_2}}(\nabla u\cdot  x+u)(\nabla u\cdot \nabla \phi)-\iint_{S_{T,r_2}}\frac{\nabla\phi\cdot x}{2}\left(|\nabla u|^2-|\partial_tu|^2+u^2\right)}_{=:\mathrm{IV}}\\
				&\underbrace{-\iint_{S_{T,r_2}}\left(\frac{\phi}{4}-\nabla \phi \cdot x\right) u^4 }_{=:\mathrm{V}}+\underbrace{\iint_{S_{T,r_2}}\frac{\phi}{2}\left(|\nabla u|^2+|\partial_tu|^2-u^2\right)}_{=:\mathrm{VI}}
			\end{split}
		\end{equation}
		
		We are going to bound $\mathrm{VI}$ from below, but to bound all other terms from above.
		
		Using Cauchy-Schwarz inequality followed by another application of Young's inequality, we obtain
		\begin{equation*}
			|\mathrm{I}|=	\left|\left(\int_{B_{r_2}}[\partial_{t}u(\phi x\cdot \nabla u+\phi u)]\big|_0^T \right)\right|\leq \tilde{C}_1 (E[\vec{u}(0)]+E[\vec{u}(T)])
		\end{equation*}
		where $\tilde{C}_1=\tilde{C}_1(R)>0$ is as in \eqref{eq:constant}. { Using \eqref{eq:id}, we further bound
		\begin{equation}\label{eq:I}
				|\mathrm{I}|\leq 2\tilde{C}_1 (A[\vec{u},0,T]+E[\vec{u}(T)]).
		\end{equation}}
		
		We next bound $\mathrm{II}$. For this, we use our assumption \eqref{condition:1} on the damping and our choice of $\phi$ together with Young's inequality to get
		\begin{equation*}
			\left|\mathrm{II}\right|\leq \frac{1}{16} \int_0^T\int_{B_{r_2}} \left(|\nabla u|^2+u^2\right)+4\Lambda_1 \tilde{C}_1^2 \int_0^T\int_{B_{r_2}}\alpha |\partial_tu|^2.
		\end{equation*}
		{
			Noting that the second integration on the right hand side is indeed bounded by $A[\vec{u},0,T]$, we thus have
			\begin{equation}\label{eq:II}
				\left|\mathrm{II}\right|\leq \frac{1}{16} \int_0^T\int_{B_{r_2}} \left(|\nabla u|^2+u^2\right)+4\Lambda_1 \tilde{C}_1 A[U,0,T].
			\end{equation}
		}
		
			We now use integration by parts to treat $\mathrm{III}$, obtaining
		\begin{equation}\label{es:spaceboundary}
			\mathrm{III}=0.
		\end{equation}
		
		To treat $\mathrm{IV}$, we use the fact $\mathrm{supp}\{\nabla \phi\}\subset B_{r_2}\setminus B_{r_1}$, Cauchy-Schwarz inequality together with the defining expression of $A[\vec{u},0,T]$ and Condition \eqref{condition:1} to obtain
		\begin{equation}\label{eq:IV}
			|\mathrm{IV}|\leq \tilde{C}_2 A[\vec{u},0,T]+ \tilde{C}_2 \int_0^T\int_{B_{r_2}\setminus B_{r_1}} \left(|\nabla u|^2+|u|^2\right)dxdt
		\end{equation}
		where $\tilde{C}_2=\tilde{C}_2(R,\Lambda_0^{-1})$ is a positive constant.
		
		In order to deal with $\mathrm{V}$, we split
		\begin{equation}\label{eq:V:split}
			\mathrm{V}=\underbrace{-\iint_{S_{T,r_2}} {\frac{\phi u^4}{4}}}_{=:\mathrm{V}_1}+\underbrace{\iint_{S_{T,r_2}}(x\cdot\nabla\phi) u^4}_{=:
				\mathrm{V}_2}.
		\end{equation}
		We use the last inequality in \eqref{eq:constant} to bound $V_2$
		\begin{equation}\label{eq:V2}
			|\mathrm{V}_2|\leq 16R\gamma \iint_{S_{T,r_2}}\phi^{\frac{7}{8}} u^4.
		\end{equation}
		Recalling $\phi=\varphi^4$, we use Gagliardo-Nirenberg inequality on the whole space to bound $\mathrm{V}_2$
		\begin{equation}
			|\mathrm{V}_2|\leq 16R\gamma\int_0^T\int_{\mathbb{R}^3}(\varphi^{\frac{7}{8}} u)^4\leq 16R\gamma B^4\int_0^T\left(\int_{\mathbb{R}^3}|\nabla(\varphi^{\frac{7}{8}} u)|^2\right)\times \left(\int_{\mathbb{R}^3}|\varphi^{\frac{7}{8}} u|^2\right)
		\end{equation}
		where $B$ is the best constant for Gagliardo-Nirenberg inequality {(see \cite{DelPino2002})}. Using \eqref{eq:constant} and the smallness assumption \eqref{eq:data:assumption} and doing some elementary calculus, we obtain
		\begin{equation}\label{eq:V2}
			|\mathrm{V}_2|\leq 32R\gamma\tilde{C}_1B^4\epsilon\left(\int_0^T\int_{B_{r_2}}|u|^2+\int_0^T\int_{B_{r_2}}|\nabla u|^2\right)
		\end{equation}
		Similarly, we apply the Gagliardo-Nirenberg inequality on the entire space to estimate $\mathrm{V}_1$ 
		\begin{equation}
			|\mathrm{V}_1|\leq \int_0^T\int_{\mathbb{R}^3}(\varphi u)^4\leq B^4\int_0^T\left(\int_{\mathbb{R}^3}|\nabla(\varphi u)|^2\right)\times \left(\int_{\mathbb{R}^3}|\varphi u|^2\right).
		\end{equation}
		Again, using \eqref{eq:constant} and the smallness assumption \eqref{eq:data:assumption}, and performing some elementary calculus, we obtain
		\begin{equation}\label{eq:V1}
			|\mathrm{V}_1|\leq 2\tilde{C}_1B^4\epsilon\left(\int_0^T\int_{B_{r_2}}|u|^2+\int_0^T\int_{B_{r_2}}|\nabla u|^2\right)
		\end{equation}
		It then follows the relationship between $r_1,r_2$ and $R$ and our choice of $\phi$ that the asserted inequality \eqref{eq:u4:1} is valid with constant being
		\begin{equation}\label{eq:c2:1}
			C_{2,1}=8\tilde{C}_1B^4.
		\end{equation}
		 
		We are now in a position to use \eqref{eq:V:split} to combine \eqref{eq:V1} and \eqref{eq:V2}, obtaining
		\begin{equation}\label{eq:V}
			|\mathrm{V}|\leq 2(64R\gamma +1)\tilde{C}_1B^4\epsilon\left(\int_0^T\int_{B_{r_2}}|u|^2+\int_0^T\int_{B_{r_2}}|\nabla u|^2\right)
		\end{equation}
		
		We use the properties of $\phi$ to bound $\mathrm{VI}$ simply as
		\begin{equation}\label{es:lowbound}
			\mathrm{VI}\geq \frac{1}{2}	\int_0^T\int_{B_{r_1}} \left(|\partial_{t}u|^2+|\nabla u|^2\right)-\frac{1}{2}\int_0^T\int_{B_{r_2}} |u|^2.
		\end{equation}
	
		We now use \eqref{eq:id:1} to combine \eqref{eq:I},\eqref{eq:II},\eqref{es:spaceboundary},\eqref{eq:IV},\eqref{eq:V} and \eqref{es:lowbound}, obtaining
		\begin{equation}\label{legend}
			\begin{split}
				\frac{1}{2}\int_0^T\int_{|x|\leq r_1}|\nabla u|^2\leq &\tilde{C}_3\Big(A[\vec{u},0,T]+E[\vec{u}(T)]\Big)+ \tilde{C}_4(\epsilon_3)\int_0^T\|u\|_{L^2(|\cdot|\leq r_2)}^2dt \\
				& \ \ \ \ +\tilde{C}_{5}(\epsilon_3)\int_0^T\int_{r_1 \leq |x|\leq r_2}|\nabla u|^2 +\tilde{C}_{6}(\epsilon_3)\int_0^T\int_{|x|\leq r_1}|\nabla u|^2 
			\end{split}		
		\end{equation}
		where
		\begin{equation*}
			\begin{split}
				\tilde{C}_3&:= 2\tilde{C}_1+\tilde{C}_2+4\Lambda_1\tilde{C}_1^2,\\ \tilde{C}_4&:=\tilde{C}_4(\epsilon_3)=\frac{9}{16}+\tilde{C}_2+2(64R\gamma+1)\tilde{C}_1\epsilon_3,\\
				\tilde{C}_{5}&:=\tilde{C}_{5}(\epsilon_3)= \frac{1}{16}+\tilde{C}_2+2(64R\gamma+1)\tilde{C}_1,\\
				\tilde{C}_{6}&:=\tilde{C}_{6}(\epsilon_3)=\frac{1}{16}+2(64R\gamma+1)\tilde{C}_1\epsilon_3.
			\end{split}
		\end{equation*}		
			In order to kick the last term on the right hand side back to the left hand side, we now fix a small number $\epsilon_3$ so that $\tilde{C}_{6}\leq 1/4$, that is
		\begin{equation*}
			2(64R\gamma+1)\tilde{C}_1\epsilon_3\leq \frac{3}{16}.
		\end{equation*} 
		With this choice of $\epsilon_3$, we absorb this last term into left hand side, obtaining
		\begin{equation}\label{eq:u:nabla:a1}
			\int_0^T\int_{|\cdot|\leq r_1}|\nabla u|^2\leq C_{2,2}\left[A[\vec{u},0,T]+E[\vec{u}(T)]+ \int_0^T\|u\|_{L^2(|\cdot|\leq r_2)}^2dt+ \int_0^T\int_{r_1 \leq |\cdot|\leq r_2}|\nabla u|^2 \right]
		\end{equation}
		with
		\begin{equation}\label{eq:c2:2}
			C_{2,2}:=4\max\left(\tilde{C}_3,\tilde{C}_4(\epsilon_3),\tilde{C}_{5}(\epsilon_3)\right).
		\end{equation}
		Using the relation of $r_1,r_2$ and $R$, we see that this implies the asserted inequality \eqref{eq:u:nabla:1} with the constant $C_{2,2}$ in place of $C_2$.
		
		Up to now, we finish the proof of the first two inequalities with constant $C_2$ being the maximum of these two in \eqref{eq:c2:1} and \eqref{eq:c2:2}. We next turn to show the last two inequalities, by choosing $\epsilon_2$ even smaller if necessary.
		
		We choose a new test function $\kappa\in C^\infty(\mathbb{R}^3,[0,1])$ satisfying $\kappa|_{|\cdot|\leq R}\equiv0,\kappa|_{|\cdot|\geq 2R}\equiv 1$, {$\sup|\nabla\kappa|\leq \beta_1:=\beta_1(R)$ and $\sup|\Delta \kappa|\leq \beta_2:=\beta_2(R)$ for some positive constants $\beta_1$ and $\beta_2$.}
		
		Put $\psi:=\kappa^4$. Then {we have
		\begin{equation}\label{eq:bound:psi}
			\sup|\Delta\psi|\leq 4\beta_2+12\beta_1=:\tilde{\beta}.
		\end{equation}
		
		}

		Recall $u$ is the arbitrary global solution satisfying \eqref{eq:data:assumption} previously chosen. Integrating the product of the equation of $u$ with $\psi u$ over $(0,T)\times \mathbb{R}^3$ and using integration by parts, we obtain 
		\begin{equation}\label{eq:id1}
			\int_0^T\int_{\mathbb{R}^3}\psi[|\nabla u|^2+|u|^2-u^4]=\int_0^T\int_{\mathbb{R}^3}\left[\frac{\Delta \psi}{2}|u|^2+\psi|\partial_t u|^2\right]-\left(\int_{ \mathbb{R}^3}\psi(\partial_t u+\frac{u}{2})u\right)\Bigg|_0^T
		\end{equation} 
		 
		 Using Cauchy-Schwarz inequality, {\color{blue}\eqref{eq:E:El}} and the defining expression of $A[\vec{u},0,T]$, we first bound
		 \begin{equation}\label{e:1}
		 	\left|\left(\int_{ \mathbb{R}^3}\psi(\partial_t u+\frac{u}{2})\cdot u \right)\Bigg|_0^T\right|\leq \tilde{C}_6 \left[E[\vec{u}(T)]+A[\vec{u},0,T]\right],
		 \end{equation} 
		where $\tilde{C}_6=\tilde{C}_6(\Lambda_0^{-1})$ is a positive constant.
		
		Thanks to our choice of $\psi$ (the supporting property of $\psi$ and the bound \eqref{eq:bound:psi}), we use Cauchy-Schwarz inequality and the defining expression of $A[\vec{u},0,T]$ to bound 
		\begin{equation}\label{e:2}
			\left|\int_0^T\int_{\mathbb{R}^3}\left[\frac{\Delta \psi}{2}|u|^2+\psi|\partial_t u|^2\right]\right|\leq \tilde{C}_7\left[A[\vec{u},0,T] +\int_0^T\|u\|_{L^2(|\cdot|\leq 2R)}^2dt \right]
		\end{equation}
		for some positive constant $\tilde{C}_7=\tilde{C}_7(\tilde{\beta},\Lambda_0^{-1})$.
		
		Applying Gagliardo-Nirenberg inequality and doing some elementary calculus with gradient estimate of $\kappa$, we use \eqref{eq:data:assumption} to conclude
		\begin{equation}\label{e:3}
			\int_0^T\int_{|x|\geq 2R}u^4dxdt\leq \tilde{C}_8\epsilon\left[\int_0^T\|u\|_{L^2(|\cdot|\leq 2R)}^2dt+ \int_0^T\int_{\mathbb{R}^3}\kappa^2|\nabla u|^2dxdt\right]
		\end{equation}
		where $\tilde{C}_8=\tilde{C}_8(\beta_1)$ is a positive constant. Using the supporting property of $\kappa$, we see that the asserted inequality \eqref{eq:u4:2} holds with constant being
		\begin{equation}\label{eq:c2:3}
			C_{2,3}:=\tilde{C}_8.
		\end{equation}
		
		Substituting the above three estimates \eqref{e:1},\eqref{e:2} and \eqref{e:3} back into \eqref{eq:id1} and using the supporting property of $\kappa$, we obtain
		\begin{equation}\label{eq:id3}
			\begin{split}
				\int_0^T\int_{|\cdot|\geq 2R} \left(|\nabla u|^2+|u|^2\right)dxdt &\leq \tilde{C}_9 \left[E[\vec{u}(T)]+A(\vec{u},0,T)\right]	+\tilde{C}_{10}\int_0^T\left\|u\right\|_{L^2(|\cdot|\leq 2R)}^2dt\\
				&\ \ \ \ \ \ \ \ \ +\tilde{C}_8\epsilon \int_0^T\int_{|\cdot|\leq 2R}|\nabla u|^2+ \tilde{C}_8\epsilon \int_0^T\int_{|\cdot|\geq 2R}|\nabla u|^2
			\end{split}
		\end{equation}
		where 
		\begin{equation*}
			\begin{split}
				\tilde{C}_9&:=\tilde{C}_6+\tilde{C}_7\\
				\tilde{C}_{10}&:=\tilde{C}_{10}(\epsilon)=\tilde{C}_{7}+\tilde{C}_8\epsilon.
			\end{split}
		\end{equation*}
		Taking $\epsilon_2$ to be even smaller (if necessary) so that 
		\begin{equation*}
			\tilde{C}_8\epsilon_2\leq \frac{1}{2},\ \mathrm{and}\ \mathrm{hence}\ \tilde{C}_8\epsilon\leq \frac{1}{2},
		\end{equation*}
		we can kick the last term on the right hand side of \eqref{eq:id3} back to left hand side, obtaining 
		\begin{equation*}
			\int_0^T\int_{|\cdot|\geq 2R} \left(|\nabla u|^2+|u|^2\right)dxdt\leq C_{2,4}\left[A[\vec{u},0,T]+E[\vec{u}(T)]+\int_0^T\|u\|_{L^2(|\cdot|\leq 2R)}^2dt\right]+C_{2.4}\epsilon\int_0^T\int_{R\leq |\cdot|\leq 2R}|\nabla u|^2
		\end{equation*}
		with
		\begin{equation}\label{eq:c2:4}
			C_{2,4}:=2\max\big(\tilde{C}_9,\tilde{C}_{10}(\epsilon_2),\tilde{C}_8\big).
		\end{equation}
		This is the asserted inequality \eqref{eq:u:nabla:2} with constant being $C_{2,4}$ in place of $C_2$. This completes the proof of the second two inequalities.
		
		Letting $C_2$ be the biggest one among \eqref{eq:c2:1},\eqref{eq:c2:2},\eqref{eq:c2:3} and \eqref{eq:c2:4} completes the proof of Lemma \ref{lem:morawetz}.
	\end{proof}

\section{Proof of Theorem \ref{thm:main:0}}	\label{sec:main}
	In this section, we are going to prove the observation inequality. Our starting point is the main estimate in Proposition \ref{prop:obs:weak} and the strategy is to suppress the second term on the right hand side of \eqref{eq:obs:week} by the first one. The key technique to execute this strategy is Ruiz's u.c.p. for wave equations on bounded domain (see Theorem \ref{thm:ruiz}).
	\begin{proof}[Proof of Theorem \ref{thm:main:0}]
		Let $\epsilon_0$, $C_0$ and $T_0$ be as in Proposition \ref{prop:obs:weak}. Take a small number $\epsilon\in(0,\epsilon_0)$ that is to be specified later on. 
		
		We argue by contradiction, assuming that the assertion does not hold for such a choice of $\epsilon$. Then, in comparison with the estimate in Proposition \ref{prop:obs:weak}, we can find a sequence of non-zero global solutions $u_n$ to \eqref{eq:kg:int} with initial data $\vec{u}_n(0)\in\mathcal{PS}^+$ satisfying 
		\begin{equation}\label{eq:assumption:small}
			E_L[\vec{u}_n(0)]\leq \epsilon
		\end{equation} 
		but
		\begin{equation}
			\frac{\int_0^T\|u_n\|_{L^2(|\cdot|\leq 4R)}^2dt}{A[\vec{u}_n,0,T]}\rightarrow_{n\rightarrow\infty}\infty
		\end{equation} 
		for each time $T\geq T_0$. It is worthy emphasizing that $T$ is independent of $\epsilon$. This divergence allows to rescale $u_n$ at least for all sufficiently large $n$ in the following way
		\begin{equation}
			w_n:=\frac{u_n}{\lambda_n}
		\end{equation}
		where 
		\begin{equation}
			\lambda_n:=\sqrt{\int_0^T\|u_n\|_{L^2(|\cdot|\leq 4R)}^2dt}.
		\end{equation}
		It then follows immediately that 
		\begin{equation}\label{eq:A:wn}
			\int_0^T\int_{\mathbb{R}^3}\alpha(x)|\partial_{t}w_n|^2dxdt \rightarrow_{n\rightarrow\infty}0,
		\end{equation}
		\begin{equation}\label{eq:wn:lower}
			1=\int_0^T\left\|w_n\right\|_{L^2(|\cdot|\leq 4R)}^2dt\leq \int_0^T\int_{\mathbb{R}^3}\left(|\nabla_{t,x}w_n|^2+w_n^2\right)dxdt,
		\end{equation}
		\begin{equation}\label{eq:wn}
			\partial^2_{tt}w_n(t,x)-\Delta w_n(t,x)+\alpha(x)\partial_tw_n(t,x)+w_n(t,x)=\lambda^2_nw_n^3(t,x).
		\end{equation}
		and
		\begin{equation}
			\lambda_n\leq \sqrt{2\int_0^TE_L[\vec{u}_n]dt}\leq \sqrt{2T\epsilon_0}.
		\end{equation}
		Thanks to this last inequality, we assume for some $\lambda\in[0,\sqrt{2T\epsilon_0}]$, there holds up to subsequences
		\begin{equation}\label{eq:lambda:cv}
			\lambda_n\rightarrow_{n\rightarrow\infty}\lambda.
		\end{equation}
		
		We next study the compactness of $\{w_n\}$. It follows from the choice of $\{u_n\}$ that
		\begin{equation}
				E[\vec{u}_n(T)]\leq C_0\left[A[\vec{u}_n;0,T]+\int_0^{T}\left\|u_n\right\|^2_{L^2(|\cdot|\leq 4R)}dt\right].
		\end{equation}
		This estimate, together with the defining expression of $w_n$, infers for all sufficiently large $n$ that 
		\begin{equation*}
			\int_0^T\int_{\mathbb{R}^3}\left(|\nabla_{t,x}w_n|^2+w_n^2\right)dxdt\leq 2CT\left[\int_0^T\int_{\mathbb{R}^3}\alpha(x)|\partial_{t}w_n|^2dxdt+1\right]
		\end{equation*}
		which gives directly via \eqref{eq:A:wn}
		\begin{equation}\label{eq:wn:H1}
			\int_0^T\int_{\mathbb{R}^3}\left(|\nabla_{t,x}w_n|^2+w_n^2\right)dxdt\leq 4CT.
		\end{equation}
		{		It then follows that $\{w_n\}$ is bounded in $H^1((0,T)\times \mathbb{R}^3)$. Thus we may assume that 
		\begin{equation}\label{eq:compactness:1}
			w_n \mathrm{~converges ~weakly~ to~ }w \mathrm{~in~} H^1((0,T)\times \mathbb{R}^3).
		\end{equation}
		}For each radius $r>0$, we use compact Sobolev embedding $H^1((0,T)\times B_r)\subset L^{\frac{4}{3}-}((0,T)\times B_r)$ to conclude that
		\begin{equation}\label{eq:compactness:2}
			w_n^3\rightarrow_{n\rightarrow\infty}w^3\ \mathrm{in}\ L^{\frac{4}{3}-}((0,T)\times B_r).
		\end{equation}  
		This enables us to use Cantor diagonal method to extract a subsequence, still denoted by $\{w_n\}$, that satisfies 
		\begin{equation}\label{eq:weak:phi}
			w_n^3 \rightarrow_{n\rightarrow\infty} w^3 \ \ \mathrm{in~the~sense~of~distributions}.
		\end{equation}
		
		In the case that $\lambda$ is strictly bigger than zero, the sequence $\{w_n\}$ enjoys better compactness result. To see this, we set $u:=\lambda w$. Then the sequence $\{u_n\}$ corresponding to $\{w_n\}$ enjoys the same convergence to $u$ with that $\{w_n\}$ does to $w$ in $H^1((0,T)\times \mathbb{R}^3)$. What's more, both sequences share the same compactness in $L^\infty((0,T),H^1(\mathbb{R}^3))$: 
		\begin{equation}\label{eq:compactness:3}
			\begin{split}
				&\mathrm{for~some~}u' \mathrm{~(resp.~}w' \mathrm{)~ in~}L^\infty((0,T),H^1(\mathbb{R}^3))\mathrm{,~the~sequence~}u_n\mathrm{~ (resp.~}w_n\mathrm{)}  \\ &\mathrm{~converges ~weakly ~to ~} u \mathrm{~(resp.~} w\mathrm{) ~in ~this ~space ~as ~} n \mathrm{~~tends ~to ~infinity.}
			\end{split}
		\end{equation}
		This follows from the fact that
		\begin{equation*}
			\{u_n\}\mathrm{~(resp.~}\{w_n\}\mathrm{)~is~bounded~in~}L^\infty((0,T),H^1(\mathbb{R}^3)),
		\end{equation*} 
		which is a direct consequence of the inequality \eqref{eq:E:El}, the fact that $E[\vec{u}(t)]$ is non-increasing in time $t$ and our assumption \eqref{eq:assumption:small}. Noting that $u=u'$ and $w=w'$ in the sense of distributions, we will only use notations $u$ and $w$ if there is no ambiguity.
		
		Therefore, we can now assume that the sequence $\{w_n\}$ satisfies the convergence relations \eqref{eq:compactness:1},\eqref{eq:compactness:2} and \eqref{eq:weak:phi}. What's more, in the case that $\lambda$ is strictly positive, it also satisfies \eqref{eq:compactness:3} with the limiting function $w$ in place of $w'$. In particular, $w\in L^\infty((0,T),H^1(\mathbb{R}^3))$. 
		
		Using convergence relations \eqref{eq:compactness:1},\eqref{eq:weak:phi} together with \eqref{eq:lambda:cv}, we see from the equations \eqref{eq:wn} satisfied by $w_n$ that $w$ is a weak solution to 
		\begin{equation}\label{eq:w}
			\partial^2_{tt}w(t,x)-\Delta w(t,x)+\alpha(x)\partial_tw(t,x)+w(t,x)=\lambda^2w^3(t,x).
		\end{equation}
		
		We are now going to show that $w$ is indeed a stationary solution. For this, we let $v:=\partial_tw$ and try to exploit the equation it obeys
		\begin{equation}\label{eq:v}
			\partial^2_{tt}v-\Delta v+\alpha \partial_{t}v +(1-3\lambda^2 w^2)v=0.
		\end{equation}
		{		Recalling that $w_n$ converges weakly to $w$ in $H^1([0,T]\times \mathbb{R}^3)$,} we may apply Fatou's lemma in \eqref{eq:A:wn} to obtain
		\begin{equation*}
			\int_0^T\int_{\mathbb{R}^3}\alpha(x)|\partial_{t}w|^2dxdt=0.
		\end{equation*}
		This implies 
		\begin{equation}\label{eq:ucp:1}
			v=\partial_tw=0\ \mathrm{on}\  (0,T)\times S
		\end{equation}
		with $S:=\{x\in \mathbb{R}^3:\alpha(x)>0\}$ and any $T>T_0$. Take an arbitrary radius $R_1>4R$. Recalling $T>T_0$ is independent of $\epsilon$, we are at free to pick a time $T>2R_1$ so that 
		\begin{equation}\label{eq:ucp:2}
			T>\mathrm{diam}(B_{R_1}).
		\end{equation}
	 	It then follows from \eqref{eq:v} that $v$ is a week solution to 
		\begin{equation}\label{eq:v0}
				\partial^2_{tt}v-\Delta v +(1-3\lambda^2 w^2)v=0 \ \mathrm{on}\ (0,T)\times B_{R_1}.
		\end{equation} 
		What's more, applying Fatou's lemma in \eqref{eq:wn:H1}, we see that $w$ is an $L^2((0,T)\times B_{R_1})$ function. Thus
		\begin{equation}\label{eq:ucp:3}
			 w \mathrm{~is ~a ~weak~} L^2((0,T)\times B_{R_1}) \mathrm{~solution~ of~ }\eqref{eq:v0}.
		\end{equation}
		The last condition we need to apply Ruiz's UCP is
		\begin{equation}\label{eq:ucp:4}
			\|(1-3\lambda^2w^2)\|_{L^\infty((0,T),L^3(B_{R_1}))}<\infty.
		\end{equation}
		We divide the proof of this result into two cases. In the case $\lambda=0$, it suffices to show $1\in L^\infty((0,T),L^3(B_{R_1}))$, which is obvious since $T$ and $R_1$ are finite. In the second case $\lambda>0$, it then follows from \eqref{eq:compactness:3} that $w\in L^\infty((0,T),H^1(\mathbb{R}^3))$. We can then use Sobolev's inequality in the spatial variable to conclude that $w^2\in L^\infty((0,T),L^3(\mathbb{R}^3))$ and hence $w^2\in L^\infty((0,T),L^3(B_{R_1}))$. Consequently, in this case, we have $(1-3\lambda^2w^2)\in  L^\infty((0,T),L^3(B_{R_1}))$ as well. This completes the proof of \eqref{eq:ucp:4}.
		
		With \eqref{eq:ucp:1},\eqref{eq:ucp:2},\eqref{eq:ucp:3} and \eqref{eq:ucp:4} at hand, we can then apply Ruiz's u.c.p. to the equation \eqref{eq:v0}, obtaining that $v$ vanishes everywhere in $(0,T)\times B_{R_1}$. Thanks to the arbitraryness of $R_1$ and the choice $T>R_1$, $v$ vanishes identically on $(0,\infty)\times \mathbb{R}^3$. Therefore, $w$ is a stationary solution satisfying 
		\begin{equation}\label{eq:w1}
			-\Delta w + w =\lambda^2w^3.
		\end{equation} 
		
		This implies $w\equiv0$. This result is to be proved by dividing into two different cases as well.
		
		In the first case $\lambda=0$, we integrate the multiplication of \eqref{eq:w1} with $w$ and use integration by parts, obtaining
		\begin{equation*}
			\int_{\mathbb{R}^3}(|\nabla w|^2+w^2)dx=0
		\end{equation*} 
		which in turn implies $w\equiv 0$.
		
		In the second case $\lambda>0$, we recall $u=\lambda w$. Substituting the expression of $u$ into \eqref{eq:w1}, we obtain 
		\begin{equation}\label{eq:u}
			-\Delta u + u =u^3.
		\end{equation}
		Integrating the multiplication of this equation with $u$ over $\mathbb{R}^3$ and using integration by parts, we obtain the equality
		\begin{equation*}
			\int_{\mathbb{R}^3}\left(|\nabla u|^2+u^2\right)dx=\int_{\mathbb{R}^3}u^4dx.
		\end{equation*}
		Denoting the left hand side of this equality by $2F$, we use Gagliardo-Nirenberg inequality to compute
		\begin{equation*}
			2F\leq C\left(\|u\|_{H^1}^{\frac{1}{2}}\|u\|_{L^2}^{\frac{1}{2}}\right)^4\leq 4CF^2
		\end{equation*}
		where $C$ is some positive constant concerning Gagliardo-Nirenberg inequality(see \cite{DelPino2002}).
		Solving this inequality, we get 
		\begin{equation}\label{eq:F}
			F\equiv 0\ \mathrm{or}\ F\geq \frac{1}{2C}
		\end{equation}
		By choosing $\epsilon$ (even smaller if necessary) to satisfy
		\begin{equation*}
			\epsilon<\min\left(\frac{1}{2C},\epsilon_0\right)\leq \frac{1}{2C},
		\end{equation*}
		we infer from \eqref{eq:F} and the relation
		\begin{equation*}
			F\leq \liminf E_{L}[\vec{u}_n(0)]\leq \epsilon
		\end{equation*}
		that $F\equiv 0$. Consequently we have $w\equiv 0$ in this case. It then follows from the strong convergence $w_n\rightarrow_{n\rightarrow\infty} w$ in $L^2((0,T)\times \mathbb{R}^3)$ and the inequality \eqref{eq:wn:lower} that 
		\begin{equation*}
			1<0
		\end{equation*} 
		which is impossible. This completes the proof of our main result.
	\end{proof}

\bibliographystyle{plain}

\end{document}